\documentclass{amsart}

\usepackage{amsmath, amssymb, amsthm,enumerate,  float}
\usepackage[pdftex]{hyperref}
\newtheorem{theorem}{Theorem}[section]
\newtheorem{lemma}[theorem]{Lemma}
\newtheorem{proposition}[theorem]{Proposition}
\newtheorem{corollary}[theorem]{Corollary}
\newtheorem*{theoremA}{Theorem A}
\newtheorem*{theoremB}{Theorem B}
\newtheorem*{conjectureC}{Conjecture C}

\theoremstyle{definition}

\newtheorem{example}[theorem]{Example}

\newtheorem*{notation}{Notation}

\theoremstyle{remark}

\numberwithin{equation}{section}

\newcommand{\ZZ}{{\mathbb Z}}
\newcommand{\RR}{{\mathbb R}}

\newcommand{\la}{\langle}
\newcommand{\ra}{\rangle}
\newcommand{\Irr}{{\mathrm {Irr}}}

\newcommand{\oo}{{\mathbf O}}
\newcommand{\Cen}{{\mathbf C}}
\newcommand{\Center}{{\mathrm {Z}}}
\newcommand{\Syl}{{\mathrm {Syl}}}
\newcommand{\AG}{{\mathcal {G}}}
\newcommand{\T}{{\mathcal {T}}}
\newcommand{\WT}{{\tilde{\mathcal {T}}}}
\newcommand{\Norm}{{\mathbf N}}

\begin{document}
\title[Real class sizes]{Groups with some arithmetic conditions on\\ real class sizes}
\author{Hung P. Tong-Viet}
\email{Tongviet@ukzn.ac.za}
\address{School of Mathematics, Statistics and Computer Science\\University of KwaZulu-Natal\\Pietermaritzburg 3209 South Africa}
\thanks{This research is supported by a Startup Research Fund from the College of Agriculture, Engineering and Science, the University of KwaZulu-Natal}
\date{\today}

\begin{abstract}
Let $G$ be a finite group. An element $x\in G$ is a real element if $x$ and $x^{-1}$ are conjugate
in $G.$ For $x\in G,$ the conjugacy class $x^G$ is said to be a real conjugacy class if every
element of $x^G$ is real. In this paper, we show that if $4$ divides no real conjugacy class sizes
of a finite group $G,$ then $G$ is solvable. We also study the structure of such groups in detail.
This generalizes several results in the literature.
\end{abstract}

\maketitle


\section{Introduction}

It is well known that the arithmetic conditions on the conjugacy class sizes have strong influence
on the structure of finite groups. For a full account, we refer the readers to an excellent survey
by Camina and Camina \cite{CC}. In particular, the problem of recognizing the solvability of finite
groups using arithmetic conditions on class sizes has attracted many authors.  A classical result
in finite group theory says that `\emph{For a fixed prime $p,$ if $p$ does not divide any conjugacy
class sizes of a finite group $G,$ then the Sylow $p$-subgroup of $G$ is an abelian direct factor
of $G.$}' (See \cite[Proposition~4]{CH} for instance). Especially, if $p=2,$ then the group $G$ is
solvable. This result has been  generalized by many authors in the literature. For example, A.R.
Camina \cite{Camina} showed that the Sylow $p$-subgroup of $G$ is a direct factor if the conjugacy
class size of every $p'$-element of $G$ is prime to $p.$  Recall that  an element $x$ in a finite
group $G$  is {\bf real} if $x$ and $x^{-1}$ are conjugate in $G.$ Equivalently, $x\in G$ is real
if and only if $\chi(x)$ is real, that is, $\chi(x)\in \RR,$ for any $\chi\in\Irr(G),$ where
$\Irr(G)$ denotes the set of all complex irreducible characters of $G.$ Of course, if $x\in G$ is
real then every element in the conjugacy class $K:=x^G$ containing $x$ is also real and we say that
$K$ is a {\bf real conjugacy class} and $|K|$ is a {\bf real class size}. It turns out that the
aforementioned results hold by restricting to real class sizes. For example, in
\cite[Theorem~6.1]{DNT}, Dolfi, Navarro and Tiep showed that the Sylow $2$-subgroup of a finite
group $G$ is normal in $G$ if all real conjugacy classes of $G$ have odd size. (For odd primes, a
similar result has been obtained by Guralnick, Navarro and Tiep \cite{GNT} but there is some
complication for the prime $p=3.$)  In \cite{NST}, Navarro, Sanus and Tiep studied the structure of
finite groups whose all real class sizes are $2$-powers. In particular, all these groups are
solvable. In this paper, we generalize a result of Chillag and Herzog \cite[Proposition~5]{CH} by
proving the following.

\begin{theoremA}
Let $G$ be a finite group. If the conjugacy class size of every odd prime power order real element
in $G$ is a $2$-power or not divisible by $4,$ then $G$ is solvable.
\end{theoremA}
This also gives a generalization to  \cite[Theorem~3.1]{DPS}. We note that the analogous problem
for complex irreducible characters does not hold as $4$ divides no irreducible complex character
degrees of the alternating group of degree $7,$ but this group is clearly not solvable.

 We now study in detail the structure of
finite groups whose conjugacy class sizes of odd prime power order real elements are not divisible
by $4.$ Note that $\oo^{2'}(G)$ is the smallest normal subgroup of $G$ such that the quotient
$G/\oo^{2'}(G)$ is a group of odd order.

\begin{theoremB}
Let $G$ be a finite group. If the conjugacy class size of every odd prime power order real element
in $G$ is not divisible by $4,$ then $\oo^{2'}(G)$ is $2$-nilpotent and $G/\oo_{2'}(G)$ is
$2$-closed.
\end{theoremB}

If $n$ is a positive integer and $p$ is a prime, then we can write $n=n_pn_{p'},$ where $n_p$ is a
power of $p$ and $n_{p'}$ is relative prime to $p.$ We will call $n_p$ the $p$-part of $n.$ One
consequence of Theorem~B is that if the $2$-part of the conjugacy class sizes of all noncentral
real elements of $G$ is exactly $2,$ then $\oo^{2'}(G)$ is $2$-nilpotent. To generalize this
result, we dare to make the following conjecture.

\begin{conjectureC}\label{conject}
Suppose that the sizes of all noncentral real conjugacy classes of a finite group $G$ have the same
$2$-part. Then $\oo^{2'}(G)$ is $2$-nilpotent. In particular,  $G$ is solvable.
\end{conjectureC}

The group $\rm{PSL}_3(2)$ shows that we cannot restrict the hypothesis of Conjecture C to
noncentral odd order real conjugacy classes since this group has only one noncentral odd order real
conjugacy class with size $7\cdot 2^3.$ Recently, it has been proved in \cite{Casolo} that if the
sizes of noncentral conjugacy classes of a group G have the same non-trivial $p$-part for some
prime $p,$ then $G$ is solvable and has a normal $p$-complement. This interesting result which
gives a weak analogue to the famous J. Thompson's Theorem on character degrees raises our hope of
the veracity of our conjecture. One  application of Conjecture C, if true, is a positive answer to
a question due to G. Navarro (see \cite[Question]{San}) which asked whether a group $G$ is solvable
if $|\Cen_G(x)|=|\Cen_G(y)|$ for any noncentral real elements $x$ and $y$ in $G.$ In other words,
whether a group $G$ is solvable if it has at most two real class sizes. As suspected, we will make
use of the classification of finite simple groups for the proof of Theorem A. It would be
interesting if one can find a classification free proof of this result.

\smallskip
\begin{notation} By a group, we mean a finite group. Let $G$ be a group.  If $g\in G,$ then we write $o(g)$ for the
order of $g.$ The greatest common divisor of two integers $a$ and $b$ is $\gcd(a,b).$ The $m^{th}$
cyclotomic polynomial in variable $q$ is denoted by $\Phi_m(q).$ We write $\Center(G)$ for the
center of $G.$ The symmetric and alternating groups of degree $n$ are denoted by $\textrm{Sym}(n)$
and $\textrm{Alt}(n),$ respectively. Finally, we denote by $\ZZ_m$ a cyclic group of order $m.$
Unexplained notation is standard.
\end{notation}

\section{Preliminaries}

 We collect here some basic properties of conjugacy classes in a finite group.
 The following results are well known and easy to verify.

\begin{lemma}\label{conjsizes} Let $N$ be a normal subgroup of a group $G.$ Then
\begin{enumerate}
\item If $x\in N,$ then $|x^N|$ divides $|x^G|.$
\item If $x\in G,$ then $|(Nx)^{G/N}|$ divides $|x^G|.$
\end{enumerate}
\end{lemma}
The next result is elementary. For completeness we present a proof.
\begin{lemma}\label{Centralizers}
Let $N\leq \Center(G)$ be a central subgroup of a finite group $G.$ If $x\in G$ such that
$(o(x),|N|)=1,$ then $|x^G|=|(Nx)^{G/N}|.$
\end{lemma}

\begin{proof}
Assume that $x\in G$ with $(o(x),|N|)=1.$ Let $H$ be the full inverse image in $G$ of
$\Cen_{G/N}(Nx).$ Then $N\leq \Cen_G(x)\leq H\leq G$ and therefore $|x^G|=|G:\Cen_G(x)|$ and
$|(Nx)^{G/N}|=|G/N:H/N|=|G:H|.$ Hence we need to show that $H=\Cen_G(x).$ As $\Cen_G(x)\leq H,$ it
suffices to show that $H\leq \Cen_G(x).$ Let $g\in H.$ We have that $Nx^g=Nx$ and hence $x^g=xz$
for some $z\in N.$ Since $x^g$ and $x$ have the same order, say $n,$ we deduce that
$(x^g)^n=x^nz^n,$ so $z^n=1.$ Therefore, $o(z)$ divides $\gcd(o(x),|N|)=1,$ so $z=1.$ Hence $x^g=x$
which means that $g\in \Cen_G(x)$ and so $H\leq \Cen_G(x)$ as required.
\end{proof}

We now turn our attention to real elements and real class sizes.
\begin{lemma}\emph{(\cite[Lemma~2.4]{DPS}).}\label{real} Let $G$ be a finite group.

\begin{enumerate}
\item If $x\in G$ is real and $|x^G|$ is odd, then $x^2=1.$
\item If $x\in G$ is real, then every power of $x$ is a real element of $G.$
\item The identity is the unique real element of $G$ if and only if $|G|$ is odd.
\item If $N\unlhd G$ and $|G/N|$ is odd, then real elements of $G$ are real elements of $N.$
\end{enumerate}
\end{lemma}

\begin{lemma}\label{realby2} If $x,g\in G$ with $x^g=x^{-1},$ then $x^t=x^{-1}$ for a $2$-element $t$ in $\la g\ra.$
\end{lemma}

\begin{proof} If $o(g)=2^km,$ where $m$ is odd, then the $2$-element $t:=g^m$  will work.
\end{proof}

The following lemma gives some conditions to guarantee that a real element of the quotient  group
$G/N$ can be lifted to a real element of the whole group $G.$

\begin{lemma}\label{lift} Suppose that $N\unlhd G$  and that $Nx$ is a real element in $G/N.$

\begin{enumerate}
\item If $(o(x),|N|)=1,$ then $x$ is real in $G.$
\item If the order of $Nx$ in $G/N$ is prime to $|N|,$ then $Nx=Ny$ for some real element $y$ of $G$
(of odd order if the order of $Nx$ is odd).
\item If  $|N|$ or the order of $Nx$ in $G/N$ is odd, then $Nx=Ny$ for some real element $y$ of $G$
(of odd order if the order of $Nx$ is odd).
\end{enumerate}
\end{lemma}
\begin{proof} These statements can be found in \cite[Lemma~3.1]{NST} and \cite[Lemma~2.2]{GNT}.
\end{proof}
As we are dealing with odd prime power order real elements, we will need the following result which
is an easy consequence of Lemma \ref{lift}(3).

\begin{lemma}\label{liftpower} Let $N$ be a normal subgroup of a group $G$ and let $p$ be an odd prime.
If $Nx$ is a $p$-power order real element in $G/N,$ then $Nx=Ny$ for some $p$-power order real
element $y$ in $G.$
\end{lemma}

\begin{proof}
Assume that $Nx\in G/N$ is a real element of order $p^a$ with $a\geq 0.$ If $a=0,$ then the result
is clear. So, we assume that $a\geq 1.$ As $p$ is odd, by Lemma \ref{lift}(3) we obtain that
$Nx=Nz$ for some real element $z$ in $G$ of odd order. It follows that $z^{p^a}\in N$ since $Nx=Nz$
has order $p^a$ in $G/N.$ Write $o(z)=p^bk,$ where $p\nmid k.$ Since $\gcd(p^a,k)=1,$ we deduce
that $p^au+kv=1$ for some integers $u$ and $v.$ Now we have that $Nx=Nz=Nz^{kv}$ as $z^{p^au}\in
N.$ By Lemma \ref{real}(2), we know that $z^{kv}$ is real in $G.$ Also $(z^{kv})^{p^b}=1$ since
$o(z)=p^bk.$ Thus $y=z^{vk}$ is a $p$-power order real element with $Nx=Ny$ as wanted.
\end{proof}
The next result describes the structure of groups with no nontrivial real elements of odd order.
\begin{lemma}\emph{(\cite[Proposition~6.4]{DNT}).}\label{NoOddOrder}
Let $G$ be a finite group. Then every nontrivial real element in $G$ has even order if and only if
$G$ has a normal Sylow 2-subgroup.\end{lemma}

We also need the following well known result due to Zsigmondy.
\begin{lemma}\emph{(Zsigmondy Theorem~\cite{Zsigmondy}).}\label{Zsigmondy}
Let $q$ and $n$ be integers with $q\geq 2,$ $n\geq 3.$ Assume that $(q,n)\neq (2,6).$ Then $q^n-1$
has a prime divisor $\ell$ such that $\ell$ does not divide $q^m-1$ for $m<n.$ Moreover $\ell\equiv
1~\mbox{\emph{(mod $n$)}}$ and if  $\ell\mid q^k-1,$ then $n\mid k.$
\end{lemma}
Such an $\ell$ is called a \emph{primitive prime divisor}. We denote by $\ell_n(q)$ the smallest
primitive prime divisor of $q^n-1$ for fixed $q$ and $n.$

\section{Real elements in simple groups}

Let $G$ be a group. We say that an element $x\in G$ is  \emph{good} if it is a real element of odd
prime power order whose class size is divisible by $4$ and the order of $x$ is prime to
$|\Center(G)|.$ Notice that if the center of $G$ is trivial, then the latter condition in the
definition above is trivially satisfied. In this section, we prove the following result.

\begin{proposition}\label{simple} Every nonabelian simple group $S$ possesses a good element.
\end{proposition}

\begin{proof} Using the classification of simple groups, we consider the following cases.

{\bf Case 1:}  $S\cong \textrm{Alt}(n),$ an alternating group with degree $n\geq 5.$ Let
$x=(1,2,3,4,5)$ and $g=(2,5)(3,4)$ be two elements in $S.$ We can see that $o(x)=5$ and
$x^g=x^{-1},$ so $x$ is a real element of order $5.$ We have that $\Cen_S(x)\cong \ZZ_5,$ if $5\leq
n\leq 6;$ and $\Cen_S(x)\cong \ZZ_5\times \textrm{Alt}(n-5),$ if $n\geq 7.$
 Thus \[|x^S|=\left\{
                                                                \begin{array}{ll}
                                                                 \frac{1}{5}n(n-1)(n-2)(n-3)(n-4) , & \hbox{if $n\geq 7$;} \\
                                                                 \frac{1}{10}n(n-1)(n-2)(n-3)(n-4), & \hbox{if $n\in\{5,6\}$.}
                                                                \end{array}
                                                              \right.
\] It follows that $4\mid |x^S|$ in any case, so $x$ is a good
element in $S.$

\smallskip
{\bf Case 2:} $S$ is a sporadic simple group or the Tits group. These cases can be checked directly
by using \cite{Atlas}.

\smallskip
We now consider the following set up. Let $S$ be a simple group of Lie type in characteristic $p$
defined over a field of size $q=p^f$ with $S\ncong {}^2\textrm{F}_4(2)'.$ Let $\AG$ be a simple,
simply connected algebraic group in characteristic $p$ and let $F:\AG\rightarrow \AG$ be a
Frobenius map such that if $G:=\AG^F,$ then $S\cong G/\Center(G).$ It follows from
\cite[Proposition~3.1]{TZ} that every semisimple element of $G$ is real unless $\AG$ is of type
$\textrm{A}_n$ with $n>1$, $\textrm{D}_{2n+1}$ or $\textrm{E}_6.$ We will use the bar convention in
$G/Z(G).$ By Lemma \ref{Centralizers}, if $x$ is a good element in $G,$ then $\bar{x}$ is a good
element in $S.$ Thus we need to find an odd prime power order real element $x\in G$ such that
$o(x)$ is prime to $|\Center(G)|$ and $4$ divides $|x^G|.$ For $\epsilon=\pm,$ we use the
convention that $\textrm{SL}_n^\epsilon(q)$ is $\textrm{SL}_n(q)$ if $\epsilon=+$ and
$\textrm{SU}_n(q)$ if $\epsilon=-;$ similar convention applies to $\textrm{GL}_n^\epsilon(q).$ We
also write ${\textrm{E}_6}^+(q)$ for $\textrm{E}_6(q)$ and $\textrm{E}_6^-(q)$ for
${}^2\textrm{E}_6(q).$

\smallskip
{\bf Case 3:} Assume that $\AG$ is of type $\textrm{A}_{n-1}$ with $n\geq 2.$ Then $S\cong
\textrm{PSL}^\epsilon_n(q)$ and $G\cong \textrm{SL}_n^\epsilon(q)$ with $n\geq 2,$ $q=p^f$ and
$\epsilon=\pm.$


$(1)$  Assume that $S\cong \textrm{PSL}_2(q),$ where $q\geq 5$ is odd. As $\textrm{PSL}_2(4)\cong
\textrm{PSL}_2(5)\cong \textrm{Alt}(5),$ we can assume that $q\geq 7.$ We know that $S$ has
dihedral subgroups of order $q\pm 1$ with cyclic subgroups of order $(q\pm 1)/2.$ Firstly, assume
that $q\equiv 1$ (mod 4). Then $(q+1)/2$ is odd and hence it has an odd prime divisor $r.$ It
follows that $S$ has a real element $x$ of order $r$ with $|\Cen_S(x)|=(q+1)/2.$ Hence
$|x^S|=q(q-1)$ which is divisible by $4$ since $q\equiv 1$ (mod 4). Finally, assume that $q\equiv
3$ (mod 4). In this case $(q-1)/2$ is odd and $q(q+1)$ is divisible by $4.$ Now let $x$ be a real
element of odd prime order $r\mid (q-1)/2$ and apply the argument above, we deduce that $x$ is
good.

$(2)$ Assume that  $n\geq 3$ and $q$ is odd. It follows from the proof of \cite[Lemma~4.4]{GNT}
that $G\cong \textrm{SL}_n^\epsilon(q)$ has a real element $x$ of $p$-power order with class size
$$|x^G|=q^{n(n-3)/2}(q+\epsilon 1)\prod_{i=3}^n(q^n-\epsilon^n 1).$$ Since
$|\Center(G)|=\gcd(n,q-\epsilon 1)$ is prime to $p,$  it suffices to show that $|x^G|$ is divisible
by $4.$ We now have that $|x^G|$ is divisible by $(q+\epsilon 1)(q-\epsilon 1)=q^2-1.$ Since $q$ is
odd, we deduce that $q^2-1$ is divisible by $4$ so is $|x^G|.$ Hence $x$ is good in $G.$

$(3)$ Assume $S\cong \textrm{PSL}_n(q)$ with $n\geq 2, q=2^f.$ Embed $\textrm{SL}_2(q)\times
\textrm{SL}_{n-2}(q)$ in $G=\textrm{SL}_n(q).$ Let $B\in \textrm{SL}_2(q)$ such that $B$ is of
prime power order and $o(B)\mid q+1.$ Define
\[x=\left(
      \begin{array}{cc}
        B & 0 \\
        0 & I_{n-2} \\
      \end{array}
    \right),
\] with $I_{n-2}\in \textrm{SL}_{n-2}(q)$ being the identity matrix. Since $B$ is a semisimple real element of
$\textrm{SL}_2(q),$ $x$ is a semisimple real element of $\textrm{SL}_n(q).$ Now for $n=2$ we have
$q\geq 4,$ so $|x^G|=q(q-1)$ is divisible by $4.$ Finally, for $n\geq 3$ we have
$|x^G|=t|B^{\textrm{SL}_2(q)}|,$ where $t$ denotes the number of decompositions of an
$n$-dimensional vector space over $\mathbb{F}_q$ into the direct sum of a $2$-dimensional and an
$n-2$-dimensional subspaces. Easy calculation shows that
\[t=\dfrac{|\textrm{GL}_n(q)|}{|\textrm{GL}_2(q)|\cdot
|\textrm{GL}_{n-2}(q)|}=q^{2(n-2)}\dfrac{(q^n-1)(q^{n-1}-1)}{(q-1)(q^2-1)},\] so $4\mid t$  and
thus $4\mid |x^G|$ as required. Furthermore, $\gcd(o(x),|\Center(G)|)=1,$ therefore $x$ is a good
element in $G.$

$(4)$ Assume $G\cong \textrm{SU}_n(q)$ with $n\geq 3$ and $q=2^f.$ Assume first that $n=3.$ Since
$\textrm{PSU}_3(2)$ is not simple, we can assume that $q\geq 4.$ Let $B\in
\textrm{SU}_2(q)=\textrm{SL}_2(q)$ such that $B$ is of prime power order and $o(B)\mid q-1.$ Embed
$\textrm{SU}_2(q)$ in $G$ and let

\[x=\left(
      \begin{array}{cc}
        B & 0 \\
        0 & 1 \\
      \end{array}
    \right).
\] Then $x$ is a semisimple real element of $G.$ We have that
$|x^G|=t|B^{\textrm{SU}_2(q)}|,$ where $t$ denotes the number of decompositions of a
$3$-dimensional non-degenerate unitary space over $\mathbb{F}_{q^2}$ into the orthogonal direct sum
of a $2$-dimensional and a $1$-dimensional subspaces and
\[t=\dfrac{|\textrm{GU}_3(q)|}{|\textrm{GU}_2(q)|\cdot
|\textrm{GU}_{1}(q)|}=q^{2}\dfrac{(q^3+1)(q^{2}-1)}{(q+1)(q^2-1)},\] so $4\mid t$  and thus $4\mid
|x^G|$ as required. Since $o(x)=o(B)\mid q-1$ and $|\Center(G)|=\gcd(3,q+1),$ we see that
$\gcd(o(x),|\Center(G)|)=1,$ so $x$ is a good element in $G.$

Assume next that $n\geq 4.$ Let $m=2\lfloor n/2\rfloor\geq 4.$ Assume first that $(n,q)\neq
(6,2),(7,2).$ Then $\ell_{mf}(2)>2$ exists. It follows from the proof of Lemma~3.5 in \cite{DNT}
that there exists a real element $x$ of order $\ell_{mf}(2)$ with $x\in Sp_m(q)\leq G$ and

$$|\Cen_G(x)|=\left\{
\begin{array}{ll}
  (q^n-1)/(q+1),                &\mbox{if } n\equiv 0~ \mbox{(mod 4);}    \\
    (q^{n/2}+1)^2/(q+1),                &\mbox{if } n\equiv 2~ \mbox{(mod 4);}    \\
      (q^{n-1}-1),                &\mbox{if } n\equiv 1~ \mbox{(mod 4);}    \\
        (q^{(n-1)/2}+1)^2,                &\mbox{if } n\equiv 3~ \mbox{(mod 4).}

\end{array}
\right.$$

Hence $|x^G|$ is divisible by $q^{n(n-1)/2}$ and so by $4.$ Moreover $|\Center(G)|=\gcd(n,q+1)$
dividing $q^2-1=2^{2f}-1$ and thus $\ell_{mf}(2)$ cannot divide $|\Center(G)|$ as $m\geq 4.$
Therefore, $x$ is good in $G.$ Assume next that $(n,q)=(6,2).$ Using \cite{GAP}, $S$ has a real
element $x$ of order $7$ with $|\Cen_S(x)|=7.$ Obviously $4$ divides $|x^S|$ and thus ${x}$ is good
in $S.$ Assume now that $(n,q)=(7,2).$ We can embed $\textrm{PSU}_6(2)$ in $\textrm{PSU}_7(2)$ and
we can choose $x$ to be a real element of order $7$ as in the previous case. In this case, we have
$|\Cen_S(x)|=63$ and so $4$ divides $|x^S|.$ Hence $x$ is good in $S.$

\smallskip
{\bf Case 4:} Assume $\AG$ is of type $\textrm{B}_n$ or $\textrm{C}_n$ with $n\geq 2.$ Then $G\cong
\textrm{Spin}_{2n+1}(q)$ or $\textrm{Sp}_{2n}(q).$ In both cases, we have
$|\Center(G)|=\gcd(2,q-1).$ Assume that $(n,q)\neq (3,2).$ By \cite[Lemma~2.4]{MT}, $G$ has a
semisimple element $x$ of order $\ell_{2nf}(p)$ with $|\Cen_G(x)|=q^n+1.$ As $\ell_{2n}(q)>2$ we
deduce that $o(x)$ is prime to $|\Center(G)|.$ We need to verify that $|x^G|$ is divisible by $4.$
If $q$ is even then $|x^G|$ is divisible by $q^{n^2}$ and so by $4$ as $n\geq 2.$ Now assume that
$q$ is odd.  In this case, $|x^G|$ is always divisible by $q^2-1.$ Since $q$ is odd, $q^2-1$ is
divisible by $4,$ so $|x^G|$ is divisible by $4.$ Now assume $(n,q)=(3,2).$ Then $G\cong S\cong
\textrm{Sp}_6(2).$ Using \cite{Atlas}, we can check that $G$ has a semisimple element of order $7$
whose class size is divisible by $4.$

\smallskip
{\bf Case 5:} Assume that $\AG$ is of type $\textrm{D}_n$ with $n\geq 4.$ Then $G\cong
\textrm{Spin}_{2n}^\epsilon(q)$ with $n\geq 4$ and $\epsilon=\pm.$ Assume first that $n$ is even
and $(n,q)\neq (4,2).$ By \cite[Lemma~2.4]{MT}, $G$  has a semisimple element $x$ of order
$\ell_{2(n-1)f}(p)$ with $|\Cen_G(x)|=(q^{n-1}+1)(q+\epsilon 1).$ As $n\geq 4,$ it is easy to check
that $|x^G|$ is divisible by $4.$ Moreover, as $|\Center(G)|=\gcd(4,q^n-\epsilon 1)$ and
$o(x)=\ell_{2(n-1)f}(p)$ is odd, we deduce that the order of $x$ is prime to $|\Center(G)|.$ Now
assume that $(n,q)=(4,2).$ Then $S\cong \textrm{P}\Omega_8^\epsilon(2)$ and the results follows
easily by checking \cite{Atlas}.

Assume that $n\geq 5$ is odd. Embed $H=\textrm{Spin}_{2n-2}^-(q)$ in $G.$ By \cite[Lemma~2.4]{MT},
there exists a semisimple element $x\in H$ of order $\ell_{(2n-2)f}(p)$ with
$|\Cen_H(x)|=q^{n-1}+1.$ Note that $x$ is real in $H$ as $n-1$ is even. So, $x$ is real in $G.$ As
in the previous case, the order of $x$ is prime to $|\Center(G)|.$ Furthermore, we have that
$|\Cen_G(x)|=(q^{n-1}+1)(q+\epsilon 1).$ It is now easy to check that $|x^G|$ is divisible by $4.$

\smallskip
{\bf Case 6:} Assume $\AG$ is of exceptional type.

$(1)$ Assume that $G\cong {}^2\textrm{G}_2(q)$ with $q=3^{2m+1},m\geq 1.$ Then $|\Center(G)|=1.$ By
\cite[Lemma~2.3]{MT}, $G$ has a  semisimple element $x$ of order $\ell_{2m+1}(3)$ with
$|\Cen_G(x)|=q-1.$ Now $|x^G|=q^3(q^3+1)$ is divisible by $q+1=3^{2m+1}+1.$ Since $3^{2m+1}+1$ is
divisible by $4,$ we deduce that $|x^G|$ is divisible by $4$ as required.

Similar argument applies to the group ${}^2\textrm{B}_2(q)$ and ${}^2\textrm{F}_4(q),$ where
$q=2^{2m+1}$ and $m\geq 1,$ with semisimple elements of order $\ell_{2m+1}(2)$ and
$\ell_{6(2m+1)}(2),$ respectively.

$(2)$ Assume $\AG$ is of type $\textrm{E}_6.$ We have that $G\cong \textrm{E}_6^\epsilon(q)_{sc}$
with $\epsilon=\pm$ and $|\Center(G)|=(3,q-\epsilon 1).$ Embed $H=\textrm{F}_4(q)$ in $G.$ By
\cite[Lemma~2.3]{MT}, $H$ has a semisimple and hence real element of order $\ell_{12f}(p)$ with
$|\Cen_H(x)|=\Phi_{12}(q).$ As $|\Center(G)|\leq 3$ and $\ell_{12f}(p)\geq 13$ by
Lemma~\ref{Zsigmondy},  we see that $o(x)$ is prime to $|\Center(G)|.$ Now it follows from the
proof of  \cite[Lemma~2.3]{MT} that $|\Cen_G(x)|=\Phi_{12}(q)(q^2+\epsilon q+1).$ Since
$\Phi_{12}(q)(q^2+\epsilon q+1)=(q^4-q^2+1)(q^2+\epsilon q+1)$ is always odd and
$|\textrm{E}_6^\epsilon(q)|$ is divisible by $4,$ we obtain that $|x^G|$ is divisible by $4.$

$(3)$ Assume $\AG$ is of type $\textrm{E}_7.$ Then $|\Center(G)|=(2,q-1).$ By \cite[Lemma~2.3]{MT},
$G$ has a semisimple element $x$ of order $\ell_{mf}(p)$ with $|\Cen_G(x)|=\Phi_{18}(q)\Phi_2(q),$
where $m=18.$ Observe first that $o(x)$ is prime to $|\Center(G)|.$ Now if $q$ is even, then
$|x^G|$ is divisible by $q^{63}$ and hence by $4.$ Otherwise, if $q$ is odd, then $|x^G|$ is
divisible by $q^2-1.$ As $q$ is odd, we deduce that $q^2-1$ is divisible by $4.$

The same argument applies to the groups $\textrm{E}_8(q), \textrm{F}_4(q),{}^3\textrm{D}_4(q)$ and
$\textrm{G}_2(q) (q>2),$ where $m=30,12,12$ and $6,$ respectively. The proof is now complete.
\end{proof}

\section{Arithmetic conditions on real class sizes}

We say that a group $G$ has property $\T$  if the conjugacy class size of every odd prime power
order real element in $G$ is  not divisible by $4,$ i.e., if $x$ is a real element of order $p^a$
for some odd prime $p,$ then $4\nmid |x^G|.$ Similarly, a group $G$ has property $\WT$ if the
conjugacy class size of every odd prime power order real element in $G$ is either a $2$-power or
not divisible by $4.$ It follows from the definitions that if $G$ has property $\T,$ then it has
property $\WT.$

\begin{lemma}\label{inherit}
Suppose that $N\unlhd G$ and that $G$ has property $\T$ or $\WT$. Then both $N$ and $G/N$ have
property $\T$ or $\WT,$ respectively.
\end{lemma}

\begin{proof}
Assume that $N\unlhd G$ and $x\in N.$ Suppose that $x$ is an odd prime power order real element in
$N.$ Then clearly $x$ is an odd prime power order real element in $G.$ Since $|x^N|$ divides
$|x^G|$ by Lemma~\ref{conjsizes}(1), we deduce that if $4\nmid |x^G|,$ then $4\nmid |x^N|;$ and if
$|x^G|$ is a $2$-power, then so is $|x^N|.$ This shows that if $G$ has property $\T$ or $\WT,$ then
$N$ also satisfies the same property. Assume now that $Nx$ is an odd prime power order real element
in $G/N.$ By Lemma~\ref{liftpower}, $Nx=Ny$ for some odd prime power order real element $y\in G.$
By Lemma~\ref{conjsizes}(2) we obtain that $$|(Nx)^{G/N}|=|(Ny)^{G/N}|$$ divides $|y^G|.$ Clearly,
if $|y^G|$ is a $2$-power or $4\nmid |y^G|,$ then $|(Nx)^{G/N}|$ is also a $2$-power or $4\nmid
|(Nx)^{G/N}|,$ respectively. Therefore, we conclude that if $G$ has property $\T$ or $\WT,$ then
$G/N$ has property $\T$ or $\WT,$ respectively.
\end{proof}

Now we prove Theorem~A which we restate here.

\begin{theorem} If $G$ has property $\WT,$ then $G$ is solvable.
\end{theorem}

\begin{proof}
We prove this theorem by induction on the order of $G.$ Assume first that $G$ has a nontrivial
normal subgroup $N$ such that $N\neq G.$ By Lemma~\ref{inherit}, both $N$ and $G/N$ are solvable by
induction hypotheses. Hence $G$ is solvable. Thus we can assume that $G$ is nonabelian simple. By
\cite[Lemma~4.3.2]{Gorenstein}, $\{1\}$ is the only conjugacy class of $G$ which has  prime power
size. Therefore, $4$ does not divide the conjugacy class size of any odd prime power order real
element in $G.$ Now Proposition~\ref{simple} yields a contradiction. This contradiction shows that
$G$ is solvable.
\end{proof}
Here is a direct consequence of the theorem above.
\begin{corollary}\label{solvableT} If $G$ has property $\T,$ then $G$ is solvable.
\end{corollary}

Let $p$ be a prime. Recall that a group $G$ is said to be $p$-closed  if some Sylow $p$-subgroup of
$G$ is normal in $G;$ and $G$ is said to be $p$-nilpotent if it has a normal $p$-complement. We
denote by $\oo^{p'}(G)$ the smallest normal subgroup of $G$ such that the quotient $G/\oo^{p'}(G)$
is a $p'$-group, that is, its order is prime to $p.$ Also $\oo_p(G)$ is the largest normal
$p$-subgroup of $G.$ Furthermore, if $G$ is solvable and $\oo_{p'}(G)=1,$ then
$\Cen_G(\oo_p(G))\subseteq \oo_p(G).$

\begin{proposition}\label{2nilpotent}
Let $G$ be a group. Suppose that $G$ has property $\T$ and that $G=\oo^{2'}(G).$ Then $G$ is
$2$-nilpotent.
\end{proposition}

\begin{proof}
Assume that $G=\oo^{2'}(G)$ and that the conjugacy class size of every odd prime power order real
element of $G$ is not divisible by $4.$ If $G$ is a $2$-group, then the conclusion is trivially
true. Hence we may assume that $G$ is not a $2$-group. Clearly the Sylow $2$-subgroup of $G$ is not
normal in $G$ since $G$ has no nontrivial factor group of odd order. By Corollary \ref{solvableT},
we know that $G$ is solvable. We will show that $G$ has a normal $2$-complement by induction on
$|G|.$ We consider the following cases.

$(1)$ Assume first that $\oo_{2'}(G)= 1.$  Let $U:=\oo_2(G)\unlhd G.$ As $G$ is solvable, by
\cite[Theorem~6.3.2]{Gorenstein} we have that $\Cen_G(U)\leq U.$ In particular, $U$ is nontrivial.
As $G$ is not $2$-closed,  $G$ possesses some nontrivial odd order real element by
Lemma~\ref{NoOddOrder}. Applying Lemma~\ref{real}(2), there is a real element $x\in G$ of odd prime
order $p.$ Let the generalized centralizer of $x$ in $G$ defined by
\[\Cen_G^*(x):=\{g\in G\:|\:x^g\in\{x,x^{-1}\}\}\leq G.\] Then $\Cen_G(x)\unlhd \Cen_G^*(x)$ and
$|\Cen_G^*(x):\Cen_G(x)|=2.$ Since \[|x^G|=|G:\Cen_G^*(x)|\cdot
|\Cen_G^*(x):\Cen_G(x)|=2|G:\Cen_G^*(x)|\] and $4\nmid |x^G|$ by our hypothesis, we obtain that
$|G:\Cen_G^*(x)|$ is odd. It follows that $P\leq \Cen_G^*(x)$ for some Sylow $2$-subgroup $P$ of
$G.$ As $U\unlhd G,$ we have that $U\leq P\leq \Cen_G^*(x).$ If $U\leq \Cen_G(x),$ then
$x\in\Cen_G(U)\leq U,$ a contradiction. Therefore, there exists an element $u\in U$ such that
$x^u=x^{-1}.$ It follows that $x^{-2}=x^{-1}x^u=[x,u]\in U.$ As $\la x\ra \cap U=1,$ we deduce that
$x^{-2}=1,$ which is impossible. This contradiction shows that $\oo_{2'}(G)$ is nontrivial.

$(2)$ Assume that $\oo_{2'}(G)\neq 1.$ Let $\bar{G}=G/\oo_{2'}(G).$ Observe that
$\oo^{2'}(\bar{G})=\bar{G}.$ Also $\bar{G}$ has property $\T$ by Lemma~\ref{inherit}. Thus
$\bar{G}$ has a normal $2$-complement by induction and so  $G$ has a normal $2$-complement as
required.
\end{proof}

We now prove Theorem~B which we restate here for the reader's convenience.

\begin{theorem}\label{2closed}
Let $G$ be a finite group. If $G$ has property $\T,$ then $\oo^{2'}(G)$ is $2$-nilpotent and
$G/\oo_{2'}(G)$ is $2$-closed.
\end{theorem}

\begin{proof}
Suppose that $G$ has property $\T.$ We first observe that $H:=\oo^{2'}(G)$ is a characteristic
subgroup of $G$ and it also has property $\T$ by Lemma~\ref{inherit}. Since $\oo^{2'}(H)=H$ has
property $\T,$ it follows from  Proposition~\ref{2nilpotent} that $H$ is $2$-nilpotent, i.e., it
has a normal $2$-complement $K.$ Since $K$ is characteristic in $H$ and $H$ is characteristic in
$G,$ we deduce that $K$ is characteristic in $G.$ In particular, $K\unlhd G$ and so $K\leq
\oo_{2'}(G).$ Let $P\in\Syl_2(H).$ As $G/H$ is of odd order, we have that $P\in\Syl_2(G)$ and
$H=PK\unlhd G.$ By Frattini's argument, we have that \[G=\Norm_G(P)K=\Norm_G(P)\oo_{2'}(G).\] Hence
$G/\oo_{2'}(G)$ has a normal Sylow $2$-subgroup as wanted.
\end{proof}

The following example shows that we cannot strengthen Proposition~\ref{2nilpotent} and
Theorem~\ref{2closed} by changing the condition that `$G$ has property $\T$' to the weaker
condition `$G$ has property $\WT$'.

\begin{example}
Let $G\cong \textrm{Sym}(4)$ be the symmetric group of degree $4.$ We can see that
$G=\oo^{2'}(G),\oo_{2'}(G)=1$ and the conjugacy class sizes of odd prime power order real elements
in $G$ are $1$ and $2^3.$ Thus $G$ has property $\WT$ but not $\T.$ Furthermore,  $G$ is neither
$2$-closed nor $2$-nilpotent.
\end{example}

\subsection*{Acknowledgment} The author is grateful to the referee for  careful reading of the
manuscript and for many helpful suggestions and corrections. The referee has shortened the proofs
of Propositions \ref{simple} and \ref{2nilpotent}. As a result, the exposition of our paper has
been improved significantly.

\end{document}